\documentclass[12pt]{amsart}
\usepackage{graphicx, times}
\pdfoutput=1

\newtheorem{Def}{Definition}

\newtheorem{Thm}{Theorem}
\newtheorem{Lemma}{Lemma}

\newcommand{\fa}{ \widehat {\pmb 1}_A}
\newcommand{\fb}{ \widehat {\pmb 1}_B}
\newcommand{\zd}{\ensuremath{\mathbb Z_{12}}}

\newcommand{\zc}{\ensuremath{\mathbb Z_{c}}}
   \newcommand{\R}{\ensuremath{\mathbb R}} 
 \newcommand{\Z}{\ensuremath{\mathbb Z}} \newcommand{\C}{\ensuremath{\mathbb C}} 

\newcommand{\dessin}[4]{
\begin{figure}[h]
\centerline{\includegraphics[width =#3]{#1}}
\caption{#2}
\label{#4}
\end{figure}}

\title{The Torii of phases}

\begin{document}
\author{Emmanuel Amiot}\thanks{manu.amiot@free.fr}
\date{5/21/2009, CPGE, Perpignan, France }


\maketitle

\bigskip

\section*{Introduction}

The present paper is concerned with the existence, meaning and use of the {\em phases} of the (complex) Fourier coefficients of pc-sets, viewed as maps from $\zc$ to \C.

Their other component, the {\em magnitude}, has received some attention already, and its meaning is more or less understood. Complex numbers are described geometrically by these two dimensions, which contrariwise to cartesian coordinates do not play permutable roles: magnitude is a length, phase is an angle according to the polar representation
$$
    z = \text{magnitude} \times e^{i\times \text{phase}} = |z|\, e^{i \arg z}
    = r\ e^{i \, \varphi}
$$

\dessin{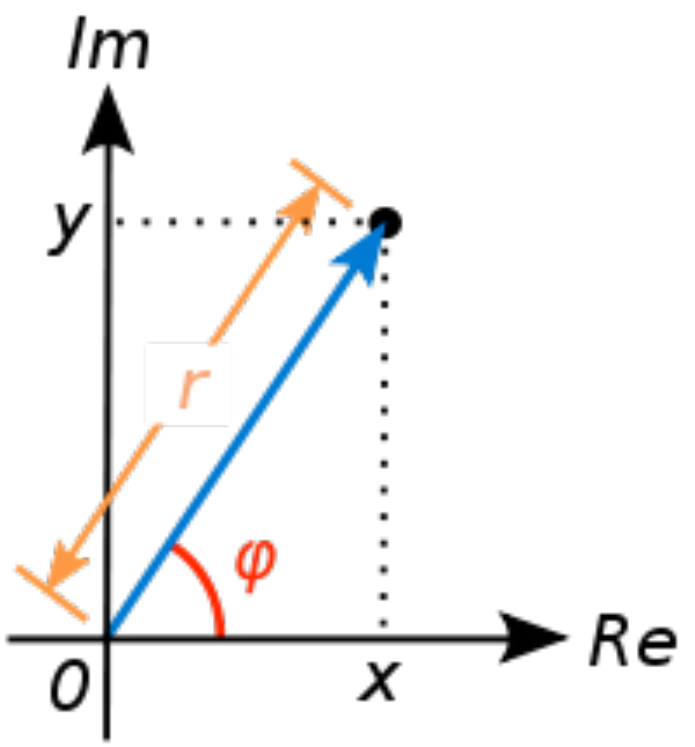}{Argument and phase of a complex number}{5cm}{complexNumber}

The first section recalls briefly the definition and useful features of one species of Discrete Fourier Transform of pc-sets (DFT for short), which is a list of Fourier coefficients, and clarifies the signification of their magnitude and phase, which may bear some relationship to perception and psycho-cognitive issues.

The second section explores a particular cross-section of the most general torus of phases (defined below), representing pc-sets by the phases of coefficients $a_3$ and $a_5$. On this 2D-torus, triads take on a well-known configuration, that of the (dual) Tonnetz which is thus equipped at last with a `natural' metric.
Some other chords or sequences of chords are viewed on this space as examples of its musical relevance. 

The end of the paper changes tack, making use of the model as a convenient universe for drawing gestures -- continuous paths between pc-sets, or even generalizations of those.

\section{From magnitude to phase}
\subsection{Discrete Fourier Transform of pc-sets}

Throughout this paper a pc-set $A\subset \zc$\footnote{$\zc$ is used in definitions for generality but in this paper all examples will be taken with $c=12$.}  is identified with its characteristic map from $\zc$ to $\C$,
$$
   \pmb 1_A : k\mapsto \begin{cases} 1 & \text{ if } k\in A \\ 0 & \text{ else} \end{cases}
$$
For  $A =\{0, 3, 7\}\subset\zd$ one would get for instance $\pmb 1_A(3) = 1$ but $\pmb 1_A(4) = 0$, i.e.
$\pmb 1_A$ takes on the following values when $k$ runs from 0 to 11:
$$
  1 , 0, 0, 1, 0, 0, 0, 1, 0, 0, 0, 0
$$ 
Later on, we may find values other than 0 or 1 (this can be construed for instance as the loudness of a particular pitch(-class) in a chord), thus vindicating the claim that $\pmb 1_A$ is just a map from $\zc$ to \C. The space of these maps, the space of `` quantities of pc's'', will be denoted as $\C^{\zc}$.

The Discrete Fourier Transform or DFT of any map $f:\zc \to \C$ will be defined as another map,
$$
   \widehat f : t \mapsto \sum_{k\in \zc} f(k) \ e^{-2i\pi k t / c}
$$
The value $\widehat f(t)$ is the $t^{th}$ Fourier coefficient of map $f$. Since the set $\C^{\zc}$ of all maps from $\zc$ to $\C$ is isomorphic to $\C^c$, the transformation $f\mapsto \widehat f$ is a linear isomorphism from this space to its image, the Fourier space\footnote{The reciprocal transformation is called inverse Fourier transform with an almost identical formula:
$$
  f(k) = \dfrac1c  \sum_{t\in \zc} f(k) \ e^{+2i\pi k t / c}
$$
 Besides, the Fourier transformation is isometrical for the hermitian metric 
$\| f \|_2 = \sqrt{\sum |f(k)|^2}$ up to a constant, according to Parceval's theorem.}.

In the aforementioned important case of pc-sets, $f = \pmb 1_A$ and the definition reduces to 
$$
   \widehat {\pmb 1}_A ( t ) = \sum_{k\in A} \ e^{-2i\pi k t / c},
$$
a sum of $c^{th}$ roots of unity.

For instance if $A = \{0, 3, 6, 9\}\subset\zd$ (a diminished seventh) one gets
$$
   \widehat {\pmb 1}_A ( t ) = 1 + i^t + (-1)^t + (-i)^t = 
   \begin{cases} 1 & \text{ if 4 divides }t  \\ 0 & \text{ else} \end{cases}
$$
since $e^{-2 i\pi 3 t/12} = e^{- i \pi\, t / 2} = (-i)^t$ and similarly for the other coefficients.
For convenience, we will denote the Fourier coefficients $a_0, a_1, É a_{c-1}$ instead of 
$\widehat {\pmb 1}_A (0), \widehat {\pmb 1}_A (1), É \widehat {\pmb 1}_A (c-1)$.\footnote{ Beware that the $a_k$'s are {\em not} the elements of the pc-set. Besides, the indexes run from 0 to $c-1$ but actually belong to the cyclic group $\zc$. This confusion is harmless and lightens the notations.}

We will need later the

\begin{Lemma} \label{symCoeff}
   For any real-valued map $f$, there is a symmetry between Fourier coefficients:
   $$
       \forall k\in\zc\quad a_{c-k} = \overline{a_k}
   $$
\end{Lemma}

Any complex number $z$ can be decomposed into real and imaginary part, but these cartesian coordinates are not particularly relevant for musical knowledge. Instead we notice that $z$ has a magnitude $|z|$ and a phase $\arg (z)$, with
$$
   z = |z|\, e^{i \arg (z)}
$$
We study these two components separately. The first is fairly well understood nowadays and the present paper will explore the second.

\subsection{Magnitude of Fourier coefficients}
As we have seen with the previous example, it may happen that all exponentials in the sum point towards the same direction (when $i^t = 1 = (-1)^t = (-i)^t$, i.e. when $t$ is a multiple of 4), yielding a maximal value. Such is always the case for coefficient $a_0= 1+1+É 1 =\#(A)$, the cardinality of the pc-set.

In his pioneering work \cite{QuinnPhD}, Quinn discovered that a pc-set with a maximal {\em magnitude} of coefficient $a_k$ among all other pc-sets with the same cardinality ($k$) sports a very special shape, being a {\em maximally even set}. For instance, any pc-set with 7 elements satisfying $|a_7| = 2 +\sqrt 3$ (this is the maximal value for 7-notes pc-sets) must be one of the twelve diatonic collections. More generally, and this is the usual purpose of Fourier analysis, the size of a given Fourier coefficient tells how much the map (or pc-set) is periodic: for instance in the above example when $A = \{0, 3, 6, 9\}$, which is 3-periodic, coefficient $a_{12/3} = a_4$ is maximal. A pc-set with comparatively large $a_4$ can be viewed as {\em minor thirdish}, like $\{0, 1, 3, 6, 9\}$ which achieves the largest value of $a_4$ among 5 notes-pc-sets.

Maximal values of other coefficients have been explored in several papers, like \cite{QuinnJMT, AmiotJMM}. \cite{Tymo2010} interprets relative magnitude in terms of a voice-leading towards a subset of an evenly divided collection, which is only to be expected since these last achieve maximal values and the magnitude is a continuous map; but Tymoczko also noticed the far from obvious fact that this magnitude is almost exactly inversely proportional to the distance to this particularly even neighbour. Quoting his paper:

\begin{quote} 
   The magnitude of a set-classÕs $n^{th}$ Fourier component is approximately inversely proportional to the distance to the nearest doubled subset of the perfectly even $n-$note set-class.
\end{quote}

Being focused on neighbours in a voice-leading sense, Tymoczko is led to `` very finely quantized chromatic universe[s]Ê'' (cf. Fig. 10 in his paper), a move towards continuous gestures and continuous spaces (see his Fig. 3). To quote him again:

\begin{quote}
Interestingly, however, we can see this connection clearly only when we model chords as multisets in continuous pitch-class space [É]
\end{quote}
However, the leap to a real continuum is not unequivocal, since he is rather more interested in microtonal divisions of the octave, e.g. $0, 12/5, 24/5 , 6/5, 18/5$. This makes sense for instance in the orbifolds popularized in \cite{orbifolds, TymoczkoGoM}, wherein any real valued pitch (modulo octave) is represented; but it must be pointed out that this blows out completely the initial setting of maps on $\zd$, which enabled to consider a Discrete Fourier transform in the first place.

\cite{Callender} neatly sidesteps this difficulty in going for the wholly continuous Fourier transform, mincing the period of the complex exponentials involved as fine as needed (for the purpose of finding the maximal possible value of the magnitude of the Fourier transform)\footnote{ This is actually reminiscent of the search for a virtual harmonic spectrum,  or of a common rational approximation to a set of real numbers, a difficult problem connected to Littlewood's conjecture.}.
Again this involves a continuous variation of pitch (class), though this time in a mathematical setting coherent with (continuous) Fourier computations, with a beautifully pedagogical exposition. Reading this paper is suggested to anyone unfamiliar with Fourier transform and its use in the study of musical structures. However, the present paper keeps stubbornly to discrete pitch classes, same as in Quinn's original setting. As we will see, this does not deter from moving to the continuous, or even from {\em mimicking} fractional transpositions.

The approach closest to ours, involving more than the magnitude (size) of Fourier coefficients of discrete pc-sets, is \cite{Hoffman}. Like the two authors quoted above, he is interested in voice-leading, and studies in detail and with mathematical rigour what happens to the values of a Fourier coefficient when one, and only one note, is changed inside a pc-set. If pitch $k$ is moved to $\ell$, the corresponding exponential in $a_t$ changes by quantity $e^{-2i\pi \ell t/c} - e^{-2i\pi k t/c}$, a vector with a finite number of possible values and fixed length\footnote{ A {\em chord} in the geometrical sense, linking vertices of a regular $c-$gon.}. This creates rhombic pictures of stunning beauty in the coordinate planes of the whole Fourier space. This geometrical approach shares the philosophy of the present paper, which can be traced back to a common origin: Quinn's search for a `` landscape of chords''\ wherein large values of some Fourier coefficient (the moutains) pinpoint some prototypical shape (as seen above, a large value of, say, $|a_4|$, denotes a `` thirdish''\ pc-set), see \cite{AmiotJMM, QuinnJMT}. The shape of a pc-set is indeed strongly related to the magnitudes of Fourier coefficients by the following theorem:

\begin{Thm} \label{homometric}
   Two pc-sets share the absolute value (magnitude) of all their Fourier coefficients if and only if they have the same intervallic content.\footnote{A proof can be found for instance in \cite{AmiotJMM}.}
\end{Thm}
The most frequent case is that of T/I related pc-sets, but this theorem also works for Forte's Z relation.

In this last publication however, Quinn states explicitely (p. 59) that he will `` throw out the direction component''. It is time  now to pick up this metaphorical gauntlet, forgetting (almost) all about the magnitude, related to the shape and intervallic structure of pc-sets; and focusing on the direction, the angular component, of Fourier coefficients.

\subsection{Phase of Fourier coefficients}

It is fairly well understood what a large $a_k$ means: it tells how well the pc-set coincides with an even division of the octave by step $c/k$. What about its direction?
As \cite{Hoffman} aptly puts it, 
\begin{quote}
The direction of a vector indicates which of the transpositions of the even chord associated with a space predominates within the set under analysis. 
\end{quote}
This can perhaps be better clarified with the following technical lemma, pinpointing the effect of T/I operations:

\begin{Lemma}\label{transp}
   Transposition of a pc-set by $t$ semitones rotates the $k^{th}$ Fourier coefficient $a_k$ by a $-2 k t \pi/c$ angle.
   
Any inversion of a pc-set similarly rotates the {\em conjugates} of the Fourier coefficients.
\end{Lemma}
\begin{proof}
  $$ \widehat {\pmb 1}_{A+t}(k) = \sum_{m\in (A+t)} e^{-2i\pi k m/c}
    = \sum_{m-t\in A} e^{-2i\pi k (m-t)/c} e^{-2i\pi k t/c}
    = e^{-2i\pi k t/c}\times \fa(k)
    $$
 Similarly for inversion, one gets for instance 
 $ \widehat {\pmb 1}_{-A}(k) = \sum\limits_{-m\in A} e^{-2i\pi k (-m)/c}= \overline{\fa(k)}$.
\end{proof}

In a nutshell, the magnitude of $a_k$ tells us something about the {\em shape} of the pc-set, about its {\em melodic} possibilities, while the phase is about {\em harmony}. Fig. \ref{transpDiatonic} shows 
the different phases of the $a_5$ coefficient for the twelve diatonic collections (since the magnitudes are identical, these coefficients move on a circle). They are rotated by $-5\pi/6$ whenever the scale is transposed by one semitone, or equivalently rotated by $\pi/6$ through a transposition by fifth, i.e. this phase is simply the position of the diatonic collection on the cycle of fifths\footnote{Notice possible the redundancies. As some notorious software publishers are wont to say, `` it's not a bug, it's a feature''. For instance, a tritone transposition of a seventh does not change coefficient $a_4$, a pleasant characteristic considering the use of such transpositions in jazz, cf. \cite{TymoczkoGoM} for a geometri discussion of this.}.

\dessin{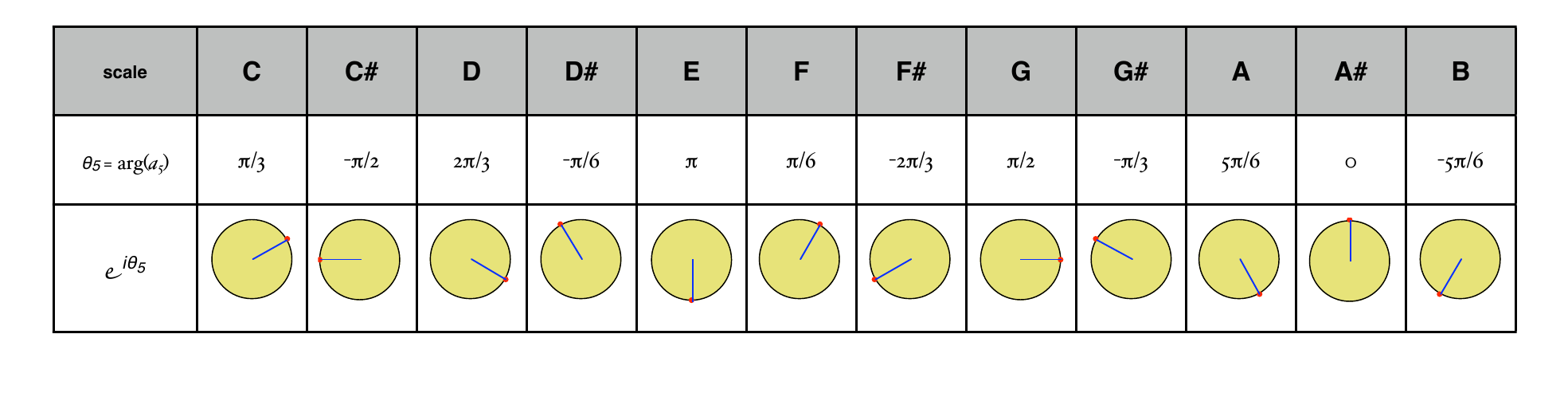}{Variations of $a_5$ for all diatonic scales}{15 cm}{transpDiatonic}

In other words, transposition of a pc-set, in the space of quantities of pcs, translates into a {\em complex rotation} in Fourier space, each Fourier coefficient being multiplied by some root of unity.\footnote{The matrix of this transformation is diagonal and unitary.}

\section{Angular position of triads}
\subsection{The torus of triads}
\dessin{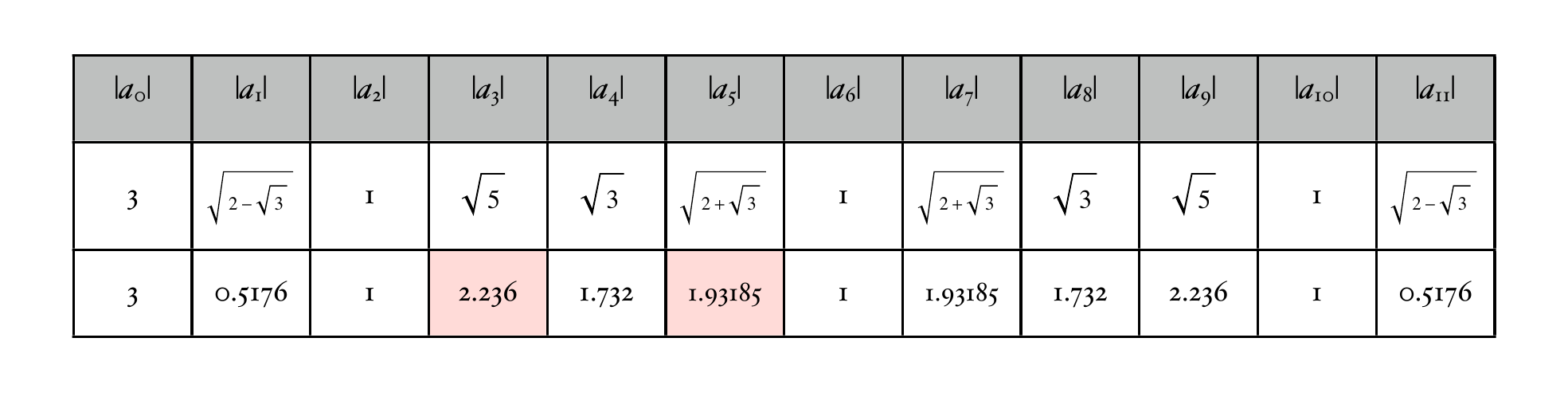}{Magnitudes of all Fourier coefficients for a triad}{15 cm}{cfTriad}

From Thm. \ref{homometric}, we know that all triads share the magnitude of their Fourier coefficients: only the phases, the angular parts, will differ. On Fig. \ref{cfTriad} we can read the values of these magnitudes\footnote{ One example to show how the table is derived: using for instance the triad $\{0, 3, 7\}$, one computes 
$$a_4 = e^0 + e^{-2 i \pi 3\times 4/12} + e^{-2 i \pi 7\times 4/12}
= 1 + 1 +  e^{-2 i \pi /3} = 2 + (-1/2) + i \sqrt{3}/2$$
 whose magnitude is
$\sqrt{ {\bigl(\dfrac32\bigr)}^2 + \dfrac34} = \sqrt{\dfrac{12}4} =\sqrt 3$.}.
This suggests two comments:
\begin{enumerate}
   \item
   The values are the same read backwards (the $0^{th}$ one excepted); this stems from Lemma \ref{symCoeff}.
   \item
    Coefficients $a_3$ and $a_5$ are the largest ($a_4$ comes close).
\end{enumerate}
This last point can be interestingly compared with \cite{Krumhansl} where the same coefficients appear prominent in a quite different setting (statistical data about the perception of a tonal environment)\footnote{ I am indebted to Aline Honingh for the connection.}.
Roughly speaking of course, it means that a triad is more fifthish and (major) thirdish than, say, chromatic or whole-tonish. Let us move to more precise notions.

Since Fourier coefficients move at different paces when a pc-set is transposed, it is difficult to visualise their movement in Fourier space. The first step is to select a region of this space where all triads can be found and easily observed. If we consider the coefficients $a_0, a_1, É a_{11}$ as coordinates in 
$\C^{12}$, all triads satisfy the following equations:
$$a_0 = 3, | a_1| = 0.5176, |a_2]=1, |a_3|=2.236, É |a_{11}|=0.5176$$
Since the condition $|z| = r$ defines a circle with radius $r$ in \C, this set of equations defines a product of circles, that is to say a {\em torus} in $\C^{12}$ (see Fig. \ref{torus}).
\dessin{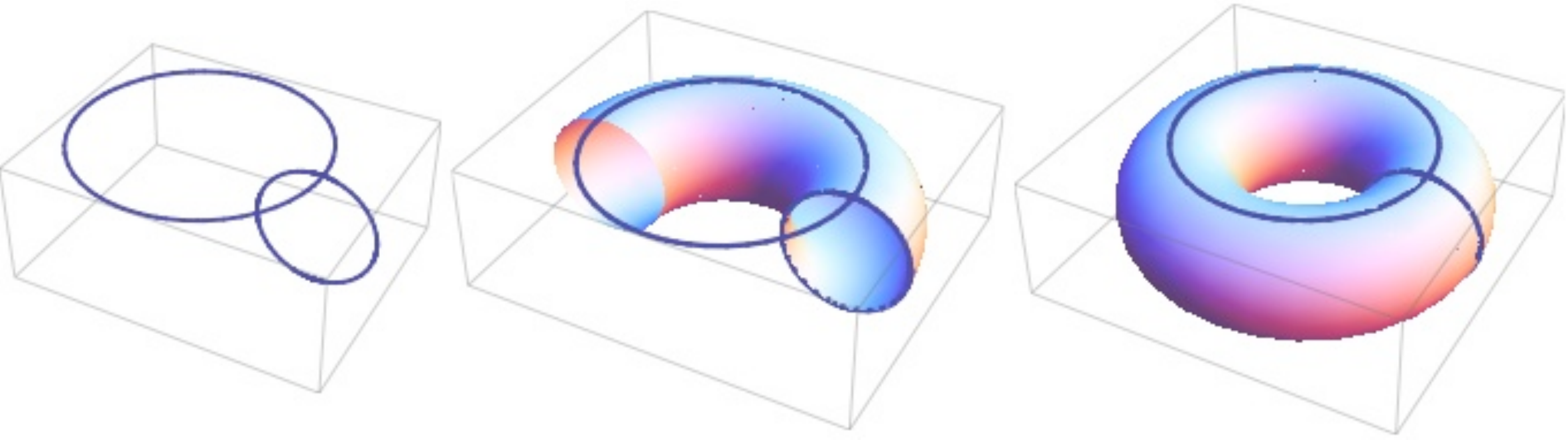}{A product of circles is a torus}{15 cm}{torus}

We can forget about coefficients $a_0$ and $a_7$ to $a_{11}$ because of redundancy. The remaining space is a 6D-torus, defined by the magnitudes of coefficients $a_1$ to $a_6$.

It is still difficult to visualise a 6D manifold in $\C^6$ alias $\R^{12}$! We must trim down this space to something more comfortable.
The values of the phases of Fourier coefficients of all triads appear on the next tables (in algebraic or numeric form).

\dessin{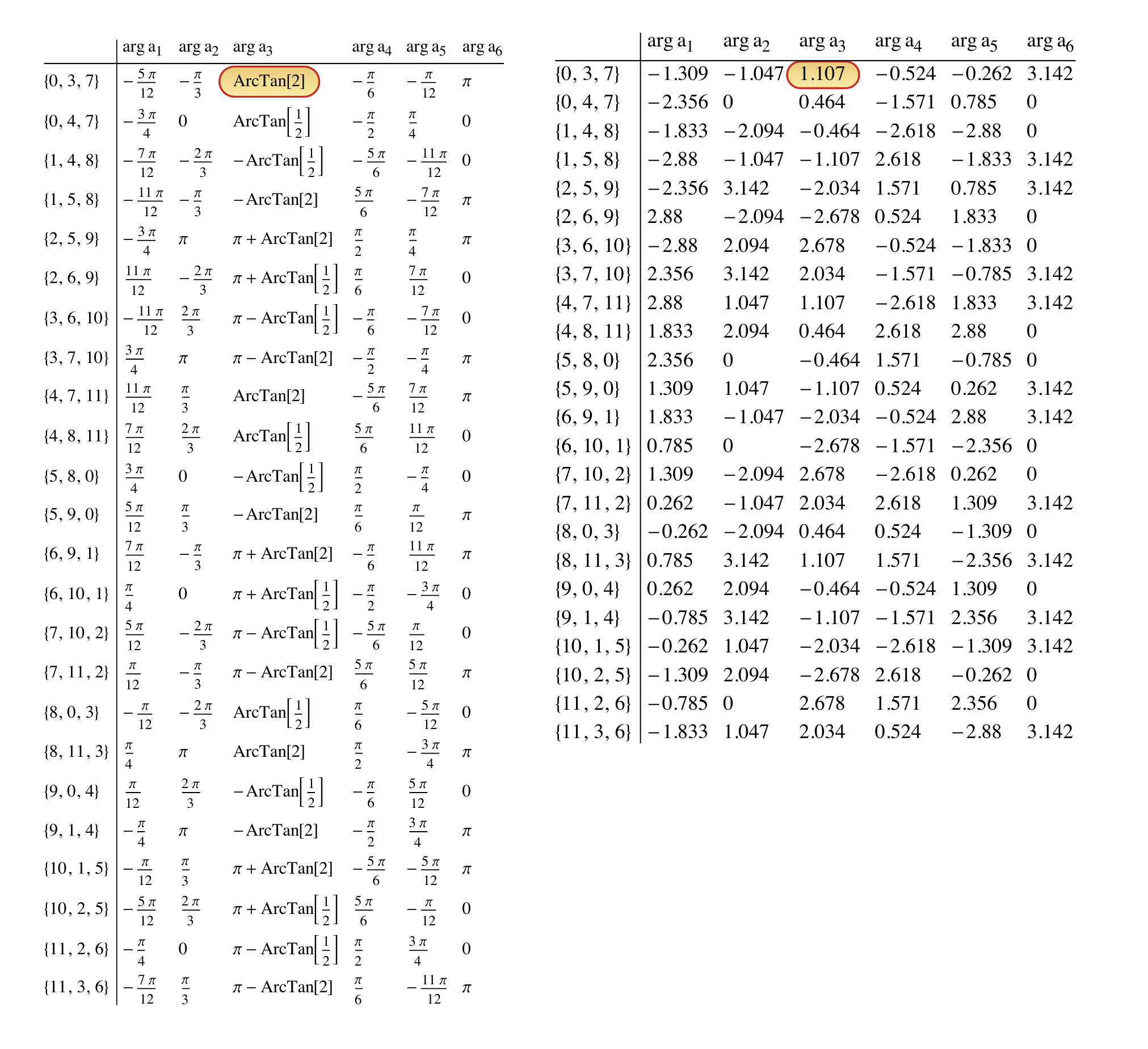}{Arguments of Fourier coefficients of triads}{18cm}{phasesTriads}
For instance, for the triad $\{0, 3, 7\}$ one computes $a_3 = 1+2i$, hence the phase is $\arctan 2$ or approximately $1.107$ (outlined on both tables).

Notice on Fig. \ref{phasesTriads} that 
\begin{itemize}
 \item
 one coordinate alone does not allow to discriminate between all 24 triads (not even between, say, major triads). For instance, major third transposition provides the same value for $a_3$ according to Lemma \ref{transp}.
 \item
 except for $\arg(a_3)$, all angles are commensurable with $\pi$.
 \item $a_6$ takes only two values, depending on the number of odd pitches in the triad.
\end{itemize}
It is feasible to retain only $a_2$ and $a_6$, allowing to distinguish between triads; but this does not make a lot of musical sense -- perhaps because these coefficients are of secondary importance for triads, as can be seen both from their magnitude and for musical reasons ($a_2, a_6$ have to do with the {\em whole-toneness} and {\em tritoneness} of pc-sets).

I selected instead coefficients $a_3$ and $a_5$, whose conjunct values are different for all 24 triads. The precise definition is the following:

\begin{Def}
  The 3-5 phase coordinates of a pc-set is the pair $(\arg (a_3), \arg (a_5))$.
  
  The 3-5 torus of triads is the $2D$ torus defined in $\C^2$ by equations
  $$
     |a_3] = 2.236 \qquad  |a_5] = 1.93185 
  $$
  and parametrized by the pair of phases defined above.
\end{Def}
In the sequel, if I mention `` the torus''\ without qualification, it will refer to that particular cross-section of the 6D overall torus of all phases.
Any point on this torus can be visualized with a pair of angles, parametrizing the standard torus $\pmb T^2$.

\dessin{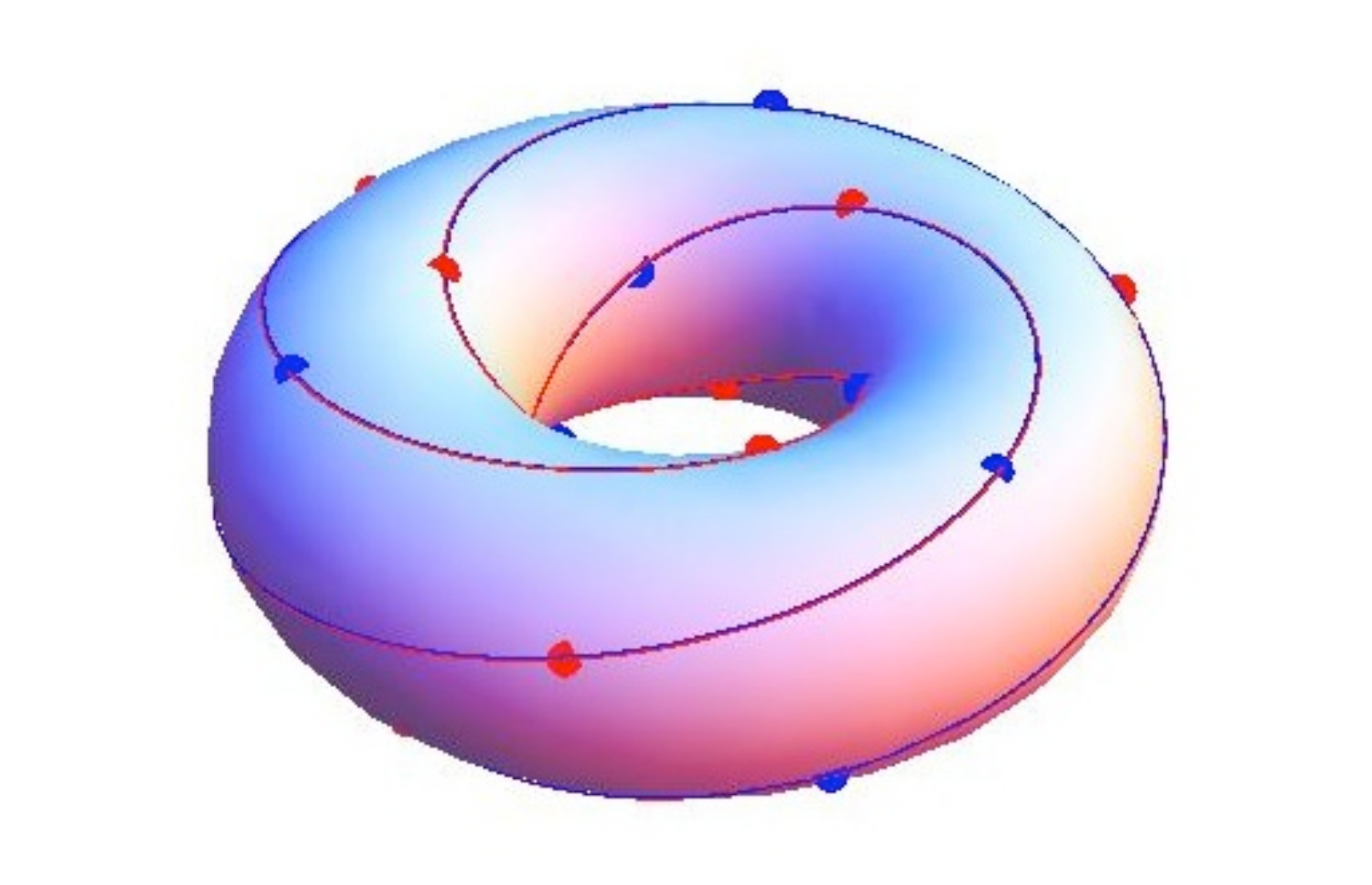}{The $3-5$ torus of triads}{9cm}{torusTriads}

On Fig. \ref{torusTriads}, major triads are the red dots and minor triads are blue. The lines connecting them will be discussed in section \ref{gestures}. For the time being, we can introduce a natural distance on this model and appreciate its musical meaning. The coordinates being angles (modulo $2\pi$), this allows to use any standard distance between pairs of real coordinates, for instance
$$
    d\bigl( (a, b), (c, d) \bigr) = \| (a-c, b-d) \|_2 = \sqrt{(a-c)^2 + (b-d)^2}
$$
This can be extended to real numbers modulo $2\pi$ by allowing $a, bÉ$ to be replaced by any representative $a+2k\pi, b+2\ell \pi É, k,lÉ \in \Z$, and retaining the minimum value\footnote{ This is the {\em  quotient metric} on $\R/2\pi\Z$.}. For instance, the distance between 
$\{0,3, 7\}$ and $\{2, 6, 9\}$ is
\begin{multline*}
    \sqrt{(1.107 - (-2.678)\mod 2\pi)^2 + (-0.262 - 1.833)^2} 
    = \sqrt{(3.786 \mod 2\pi)^2 + 2.095^2} \\
    = \sqrt{(3.786 - 2\pi)^2 + 2.095^2} = \sqrt{2.5^2 + 2.095^2} = 3.26 
\end{multline*}

This is a very physical measurement of the distance between two points on the torus Fig. \ref{torusTriads}, which can be obtained with a ruler if the torus is carved and unfolded on a table, see Fig. \ref{torusUnfolded}. This picture must be understood as illimited, each side of the picture being glued to the opposite one.

\dessin{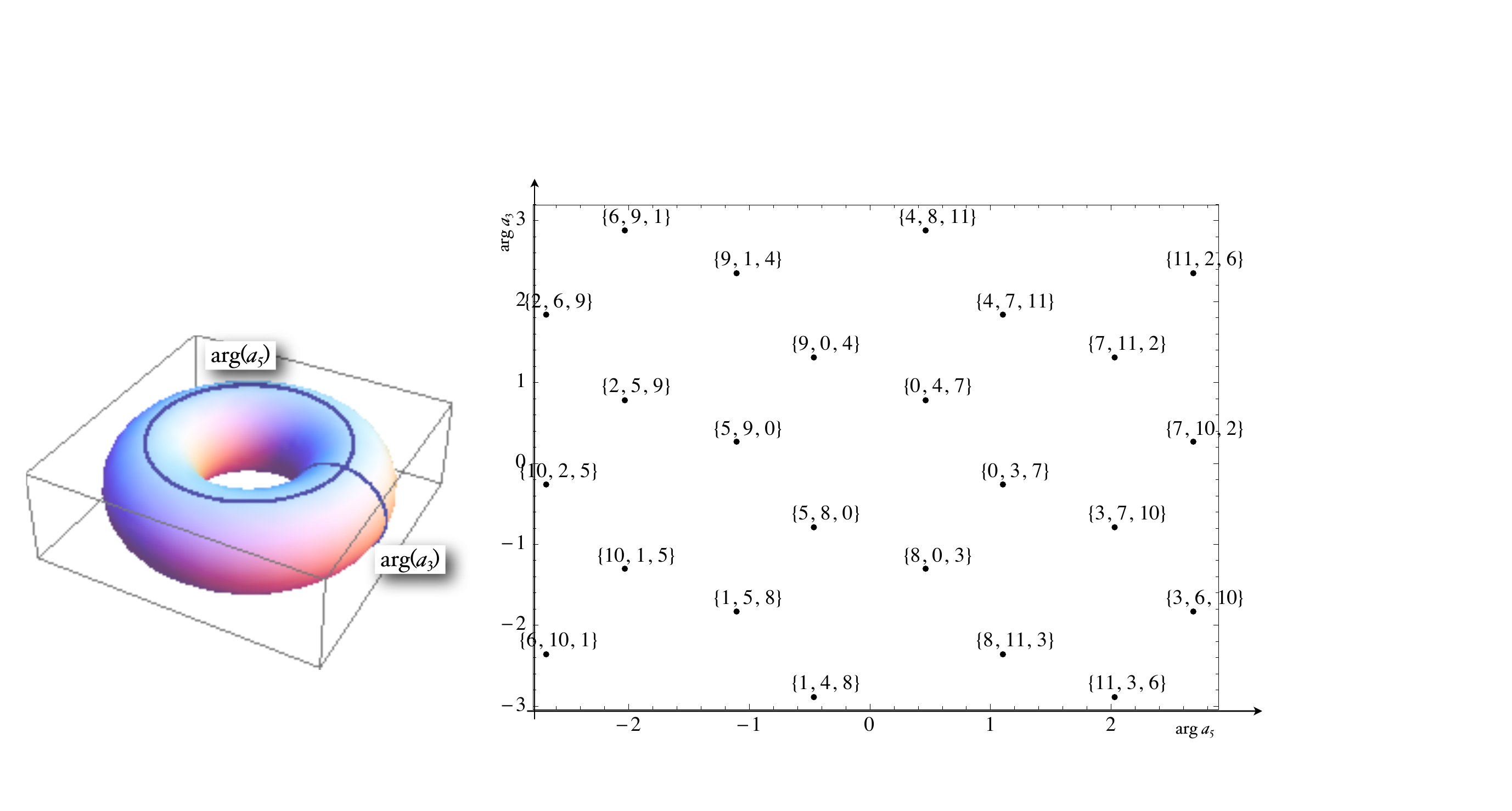}{The torus of triads unfolded}{15cm}{torusUnfolded}

This picture may look familiar to some readers, since the relative disposition of triads is the same as in the dual Riemannian Tonnetz! this can be checked on the distance table Fig. \ref{tableDistances}, wherein we see that the immediate neighbours of (say) C major e.g. $\{0, 4, 7\}$ are its LPR transforms, E minor, A minor and C minor.

\dessin{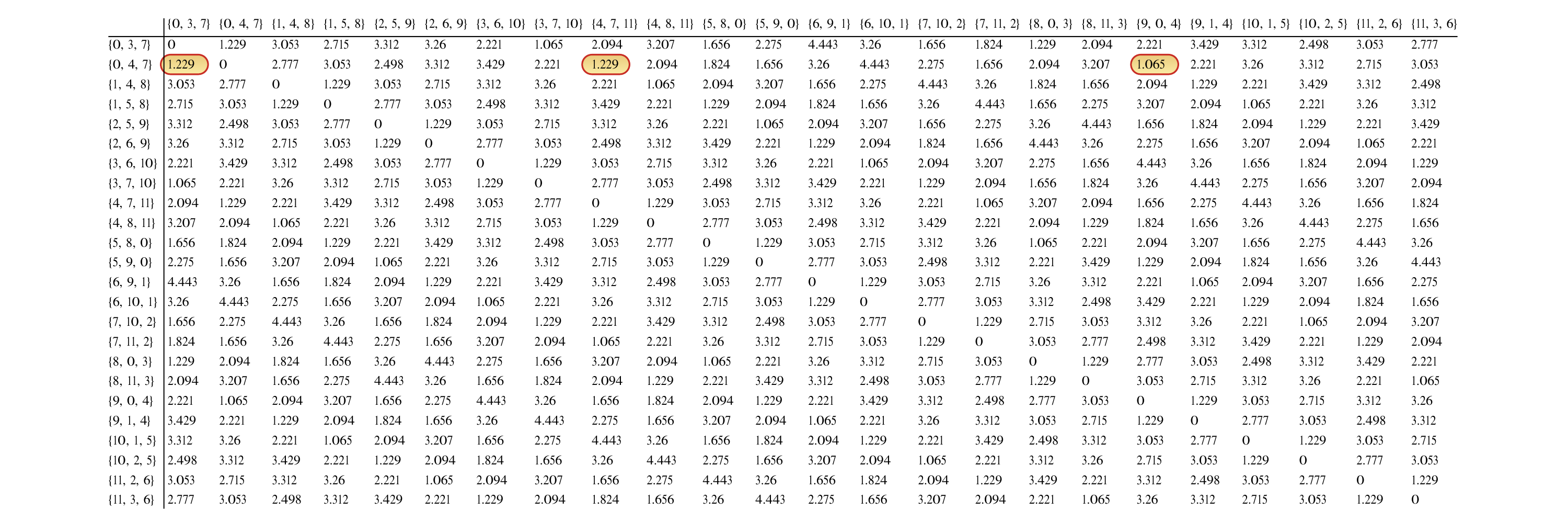}{Angular distances between triads on the torus}{18cm}{tableDistances}

This result was not a goal of my research (nor even expected, if truth be said); it does however vindicate the musical pertinence of the torus model. A slightly distorted picture could be drawn, making the distances to all three neighbours identical by stretching the canvas, with the formula
$$
   d(A, B) =\| \bigl(   0.7365 (\arg(a_5) - \arg(b_5)), \arg(a_3) - \arg(b_3)\bigr) \|_2
   \qquad
$$
where $(a_3, a_5)$ (resp. $(b_3, b_5)$) are the Fourier coefficients of triad $A$ (resp. $B$).

This provides a satisfying answer to TymoczkoÕs pessimistic though well-argued 
comment in \cite{TymoczkoMTO}: 
\begin{quote}
Thus, neither voice leading nor common tones allow us to characterize Tonnetz distances precisely. We 
seem forced to say that Tonnetz-distances represent simply the number of parsimonious moves needed 
to get from one chord to anotherÑand not some more familiar music-theoretical quality. 
\end{quote}
This should not come as a complete surprise (yes, it is easier to be wise after the facts), since the 
geometry of the Tonnetz involves neighbours one (major) third or one fifth away, closely related to the Fourier coefficients involved. It might also bear some relationship with the aforementioned study by \cite{Krumhansl} about the perception of pitch classes in a tonal environment, though it is still  early to postulate a direct perception of (something equivalent to) Fourier coefficients of musical structures in the brain.

\subsection{What appears on the torus and what does not}
As suggested by the very different mathematical properties of magnitude and phase of Fourier coefficients (see above), they pinpoint equally different musical qualities of pc-sets. Since the phase is related to which translate of a prototypical chord best coincides with the given pc-set, it tells about harmonic relationships, not voice-leading. We may therefore feel elated, when noticing that in the case of triads, the topology found on the torus is that of the dual Tonnetz, where the neighbours of a triad are those reached by parsimonious movements. One deep explanation of this relationship between harmonic and voice-leading moves can be found in \cite{TymoczkoGoM} where the author discusses the effect of transposition on almost equal divisions of the octave, such as triads. Indeed, the perfectly equal divisions, augmented triads, appear in some theoretical models and are discussed in the light of the torus in subsection {\bf `` Other triads''} below.
Another rather trivial correlation with voice-leading is the {\em continuity} of the torus coordinate system\footnote{ The map $A\mapsto (\arg a_5, \arg a_3)$ is smooth wherever it is defined.}: in non mathematical words, a small change in the pc-set (such as  moving one note by one semitone) induces a small change in the Fourier coefficients. This depends on which coefficient however, see \cite{Hoffman} for a thorough study of such moves.\medskip

I borrow an illuminating example from \cite{Callender}:  in Fig. \ref{QRS}, chord R is close to S in voice-leading terms, and harmonically close to chord Q. Indeed the Fourier coefficients' phases are close for Q, R and stand apart for R and S.

\dessin{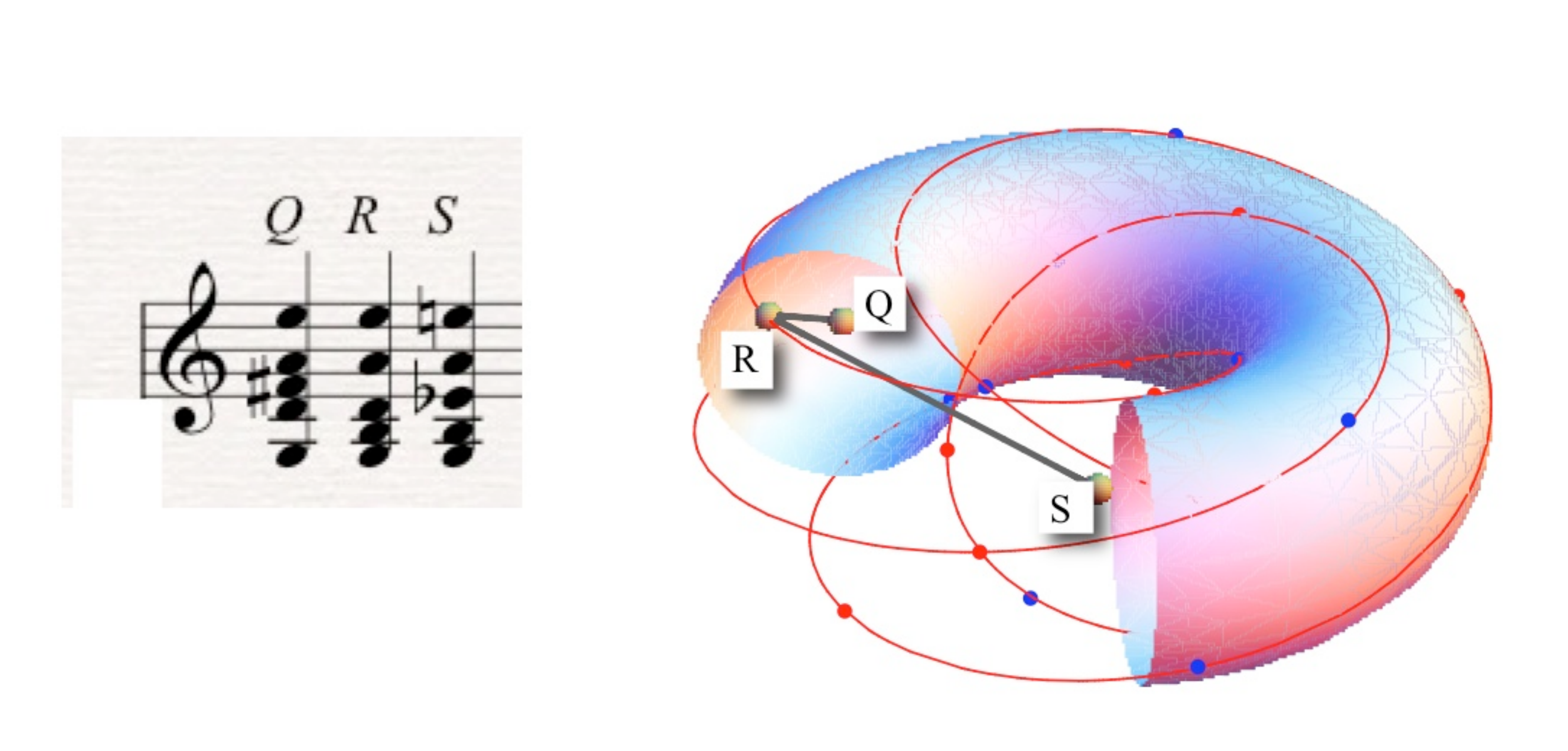}{Angular distances between complex chords}{15cm}{QRS}

All in all, a good correlation between phase distance and voice-leadings should be seen as more than a coincidence, but less than mandatory. However, when taken together with the correlation established by \cite{TymoczkoJMT} between magnitude of coefficients and neighbourhood to special pc-sets, it strengthens the case of Fourier coefficients in the study of voice-leading.
\newpage

\section{The continuous torus}
\subsection{Gestures}\label{gestures}
In \cite{Mazzola} Guerino Mazzola introduced\footnote{ The authors even mentioned Fourier spaces as a natural example of a domain for such gestures, though without getting particular.} a topological formalism for {\em gestures}, i.e. continuous paths between discrete objects or events (say from one triad to another), answering a question asked by Lewin in \cite{Lewin}: what happens between $s$ and $t$ ?

Here we do not follow Mazzola in his subsequent explorations of the general notion of gestures, but focus on a simple case: we will explore the torus model as a container for natural paths between triads, or other pc-sets. Let us begin by recalling how one transposes a triad, in phase coordinates: say $A$ is a major triad, with angular coordinates $\arg a_5, \arg a_3$ and set $B = A + t$, a transposition by $t$ semi-tones; then the phase coordinates of the new triad are
$$
   \arg b_5 = \arg a_5 - 5 \pi\, t/6 \qquad  \arg b_5 = \arg a_5 -  \pi\, t/2  
$$
according to Lemma \ref{transp}.

Now if we allow $t$ to vary continuously, rotating both Fourier coefficients (albeit with different speeds), the point described by these coordinates will draw a line on the torus, which includes all major triads once when $t$ varies between 0 and 12. It is the red line on Fig. \ref{torusTriads}, the blue line being the same thing for minor triads.

\dessin{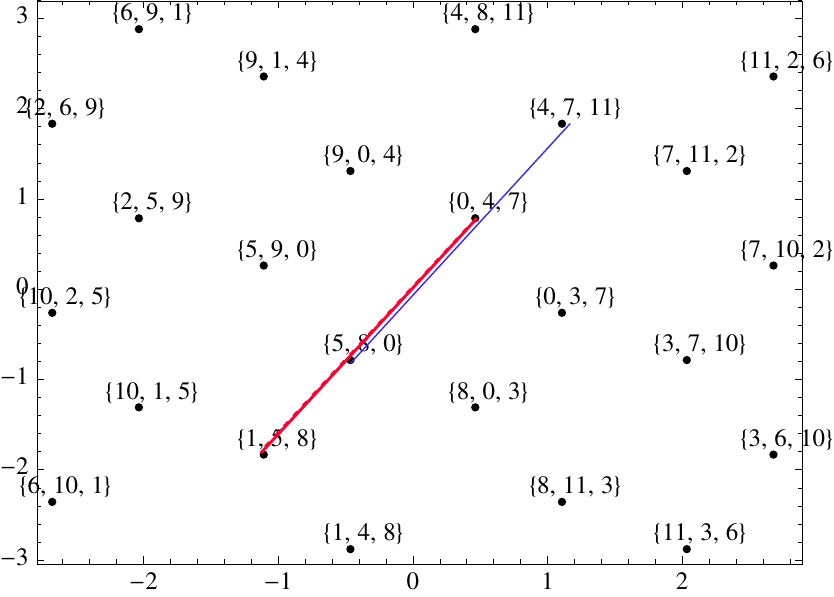}{Major triads are close to minor triads}{10cm}{unfolded}

This line is non trivial in many ways:
\begin{enumerate}
 \item
 It is a geodesic, the shortest line between all major triads (a straight line on the torus unfolded, see Fig. \ref{torusUnfolded}).
 \item
 It cannot be contracted to a point (the fundamental group of a Torus is non trivial!).
 \item 
 It is very close to the blue line of minor triads.
\end{enumerate}
This last coincidence, if it may be called that, deserves another look. First we must make sense of the points on the line which are {\em not} triads. Keeping as many parameters unchanged as possible (other Fourier coefficients, the magnitude of $a_3$ and $a_5$) one can reverse the Fourier transform and obtain a map in the original space $\C^{\zd}$ of `` quantities of pcs''. These maps no longer characterize genuine pc-sets, because their values, albeit real, are not 0's or 1's anymore\footnote{Actually one can use a transversality argument to prove that, generically, moving away from a pc-set introduces some {\em negative} coordinates in pc-space.}. A good example is the closest point to F minor on the red path of major triads, cf. Fig. \ref{unfolded} with a close-up on the unfolded Torus.

Precise measurements\footnote{ See table \ref{tableDistances}.} show that the $\arg (a_5)$ coordinate of F minor ($\{5, 8, 0\}$) is exactly at $60\%$ between $C$ and $C\#$ major ($\{0, 4, 7\}$ and $\{1, 5,  8\}$), but on the vertical dimension there is a $2\%$ difference (on the $\arg (a_3)$ coordinate). Rotating the coefficients  of C major  by $-0.056$ radiants ($\times 5$ for $\arg (a_5)$, $\times 3$ for $\arg (a_3)$ according to the formula above) and applying Inverse Fourier transform yields the following map in $\C^{\zd}$, or distribution of the 12 pitch classes:
$$
\{ 1.0087, 0.0840, -0.00806, -0.0530, 0.943, 0.164, -0.139, 0.999, 0.118, -0.0733, 0.0776, -0.120 \}
$$
I will not ask the reader to stretch his credibility to the belief that this is a pc-set. Still, it {\em is} close to
$$
 \{1, 0, 0, 0, 1, 0, 0, 1, 0, 0, 0, 0\} = F\ \text{minor}.
$$
as can be seen on Fig. \ref{fakeFminor}.

\dessin{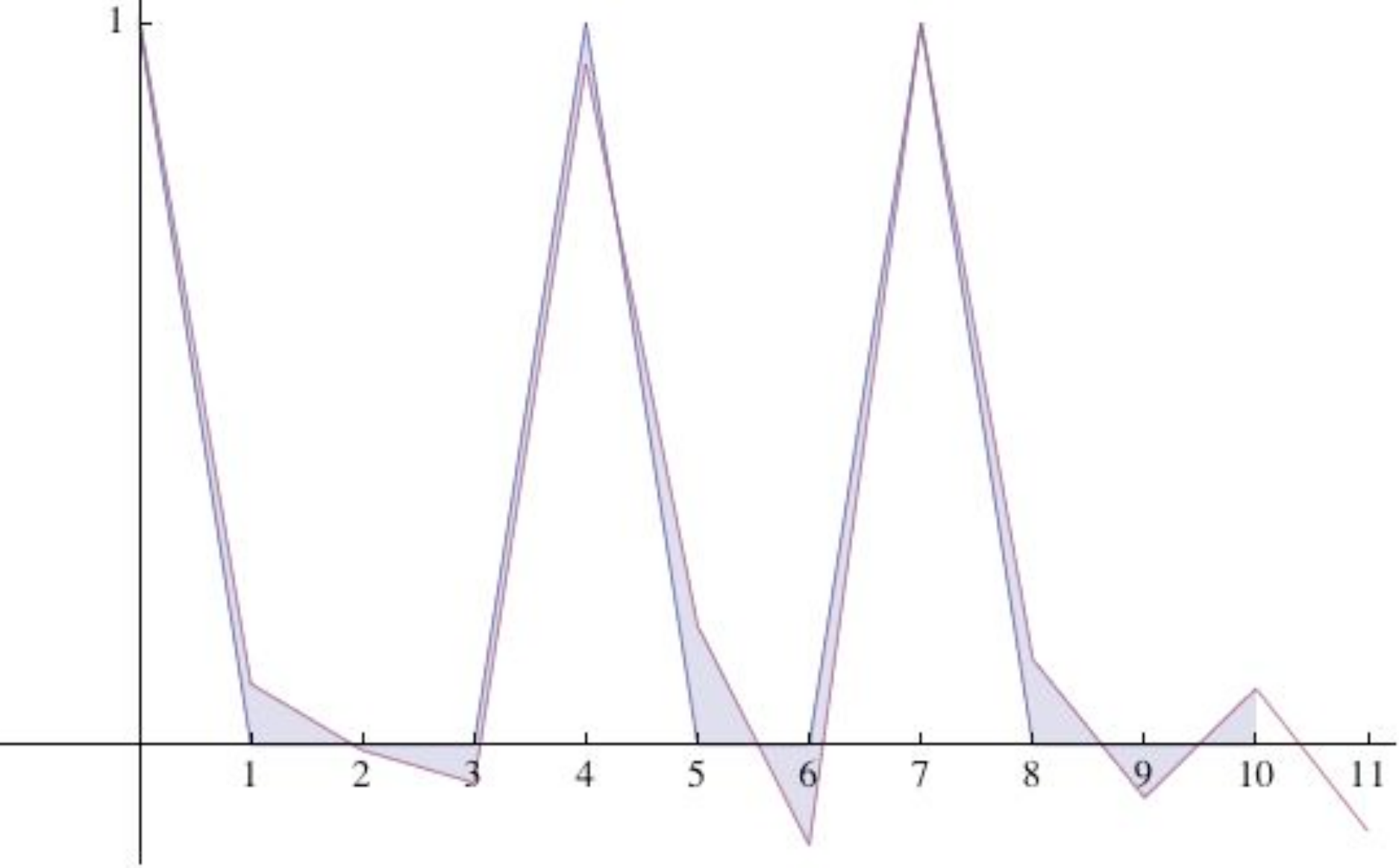}{F minor and its closest neighbour on the red line}{8cm}{fakeFminor}
Despite the fact that we are still working inside $\C^{\zd}$ and not in an orbifold or any space with continuous distribution of pitch, we have managed to mimick a transposition of C major by 0.57 semitone. Unfortunately this paper cannot include sound files (there is one on my home page), but such a transposition (in physical space) does indeed provide a good approximation of F minor\footnote{ It must be pointed out that the red and blue lines are close {\em on this particular torus}, with phases of coefficients $a_3$ and $a_5$. Most other coefficients are much further apart, especially $a_1$}.

\subsection{Splitting the inversion and other exotic gestures}
Transposition, discrete or continuous, can be done quite easily in the torus model.
It is more difficult and interesting to move continuously between a pc-set and one of its inverses. \cite{Peck} explored the question of splitting the inversion in several submoves, and found surprisingly esoteric answers. His approach consisted in solving the equation $r^k = I$, where $I$ is an inversion, $r$ is the looked-for $k^{th}$ root of $I$ in some group, usually a subgroup of the permutation group or one of its linear representations. A puzzling fact is that what works between, say, C major and C minor does not work anymore between other triads\footnote{ This is probably related to Mandereau's stunning result in \cite{Homometry}, expressing that homometry is not a group action.}. I will develop another theory below, wherein it is possible to split the inversion into infinitely many submoves, down to a continuous gesture. The seminal fact is the following:

\begin{Thm}\label{spectralUnit}
    Pc-sets $A$ and $B$ have the same interval content if, and only if, there exists a map $u\in\C^{\zc}$ such that
    $$
       \fa \times \widehat{u} =\fb \quad\text{with}\quad \forall k\in\zc,\ |\widehat{u}(k)| = 1
    $$
    ($\times$ denotes termwise multiplication)
\end{Thm}
This is really the same as Thm. \ref{homometric}, introducing maps with unit Fourier coefficients. This is explained with more details\footnote{ The gist of the proof is that $| \fa (t)| =  | \fb (t)|$ for all $t$, hence there exists a complex number with unit length $\xi(t)$ such that $ \fa (t) =  \xi(t) \times \fb (t)$, even when the Fourier transform vanishes. } in \cite{Homometry} and alternative (matricial) models are described in \cite{AmiotSethares}, allowing the computation of $u$ without using Fourier and inverse Fourier transform.

\begin{Def}
  Such a map $u\in\C^{\zc}$ with $\forall k\in\zc,\ |\widehat{u}(k)| = 1$ is called a {\em spectral unit}.
\end{Def}
The set of spectral units is thus isomorphic, via the Fourier transform, to a product of unit circles, which endows it with a group structure. Moreover, any spectral unit has several $k^{th}$ roots, for any value of $k$: if the Fourier coefficients of $u$ are $\widehat u (0)= e^{i\alpha}, É \widehat u (c-1) = e^{i\omega}$, then the map $v$ with Fourier coefficients $\widehat v (0)= e^{i\alpha/k}, É \widehat v (c-1) = e^{i\omega/k}$ is a $k^{th}$ root of $u$.

The red and blue lines of major/minor triads are instances of such gestures, orbits of the action of the continuous subgroup  
$$
     G_{\widehat u} = \left\{(1, e^{-2i\pi\,t/12}, e^{-4i\pi\,t/12}, É e^{-2k i\pi\,t/12}, É
     e^{-22 i\pi\,t/12}) ;       				\ t\in\R \right\}
$$
i.e. restricted to the 3-5 torus, phases with form $(-\pi\,t/6, -  \pi\,t/2)$.

More surprisingly, the spectral unit that connects C major ($\{0, 4, 7\}$) with A minor ($\{0, 4, 9\}$) is defined by
$$
  \widehat u = (1, -\dfrac{\sqrt 3}2 + \dfrac i 2, -\dfrac 1 2 + i\dfrac{\sqrt 3}2,
    \dfrac 3 5 - \dfrac{4i}5, \dfrac 1 2 + i\dfrac{\sqrt 3}2, \dfrac{\sqrt 3}2 + \dfrac i 2, 1, É)
$$
which is of {\em infinite order} (because the third Fourier coefficient 
$\widehat u(3) =\dfrac 3 5 - \dfrac{4i}5$ has infinite order, its phase being incommensurable with $\pi$) 
in the group of spectral units. Since the fifth coefficient has order 12, applying the associated transformation beyond A minor yields an infinite sequence of points on the 3-5 torus that is {\em dense} in 12 circles, each characterized by one triad. These pc-sets have Fourier transforms equal to $\widehat u^k \times \fa, k\in\Z$, where $A$ is the C major triad. This orbit is suggested on Fig. \ref{infiniteOrbit}. A movie is available on my web page.

\dessin{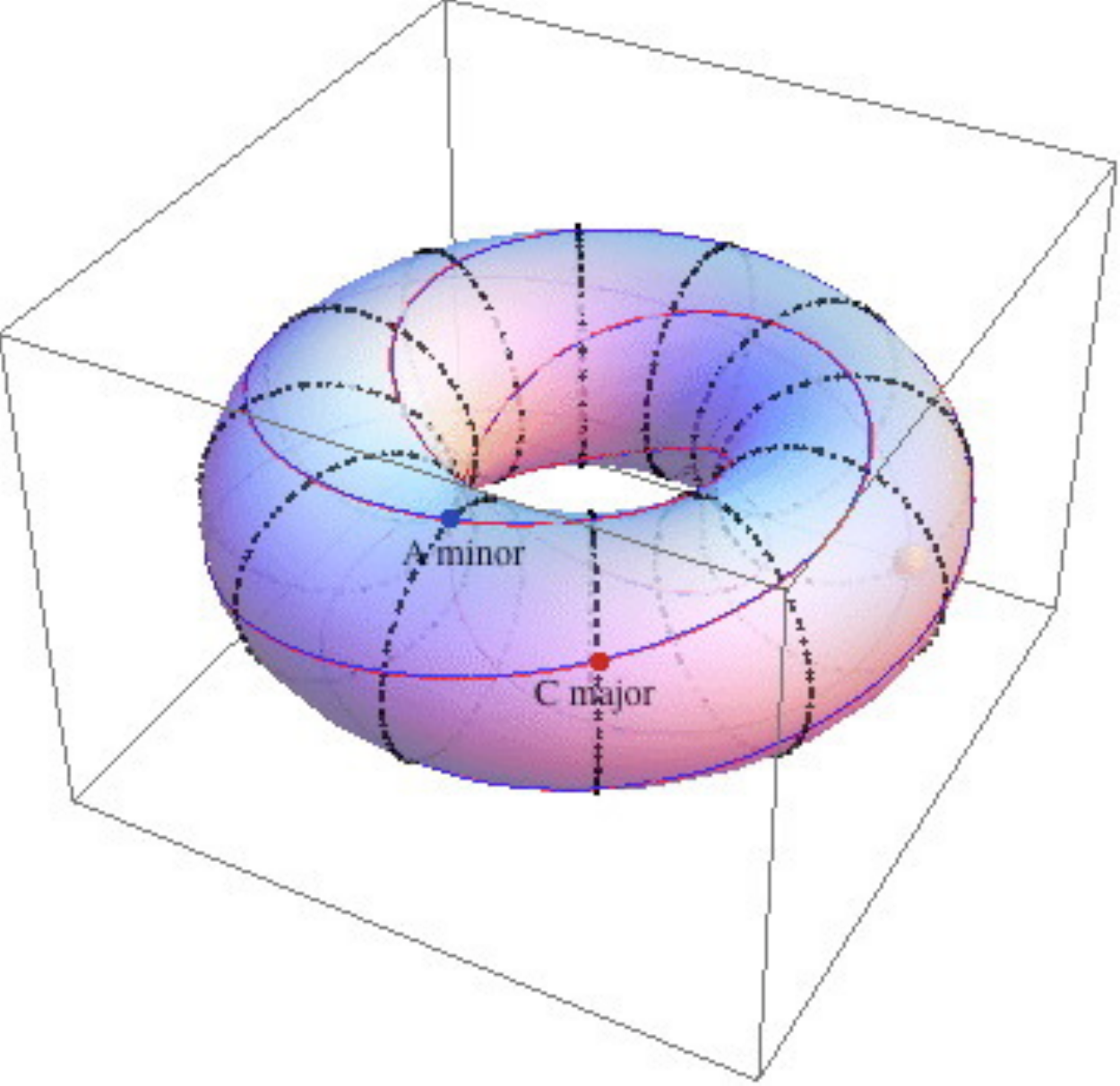}{Inversion beyond A minor}{11 cm}{infiniteOrbit}

Here I did not try to interpolate an infinite gesture, quite the reverse -- the discrete iteration already yields a suggestion of the continuous. 
\newpage

Let us end with connecting Z-related (not T/I related) pc-sets. The famous all-interval pc-sets $\{0,1,4, 6\}$ and $\{0,1,3,7\}$ are connected by the spectral unit defined by
$$
   \widehat u = \bigl(\dfrac{\sqrt 3}2 - \dfrac i 2, \dfrac 1 2 + i\dfrac{\sqrt 3}2, -i, 1, 
   \dfrac{\sqrt 3}2 - \dfrac i 2, 1, É \bigr)
$$
It has order 12 in the group. But the transforms are not all pc-sets. We get 9 genuine pc-sets and 3 generalized distributions, cf. Fig. \ref{Zrelated}.

\dessin{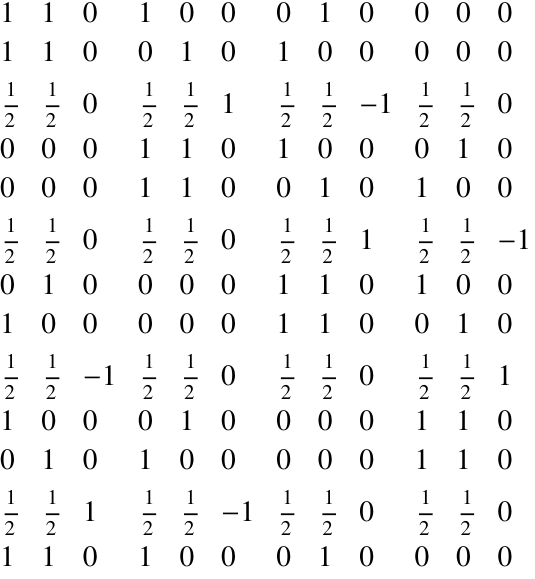}{(0 1 3 7) to (0 1 4 6) toÉ}{6 cm}{Zrelated}

This time it would be quite feasible to draw a continuous line between $\{0, 1, 3, 7\}$ and $\{0, 1, 4, 6\}$, though I think it is more interesting to iterate the transform and see how it bounces back to pc-sets (minor third transposes of the initial ones) part of the time. The same transformation applied to C major would yield its transposes by minor third -- 4 genuine pc-sets -- and 8 generalized distributions. This transformation effects a cubic root of the minor third transposition, but it is not a semitone transposition.
\medskip

We have been splitting movements between obviously related pc-sets in some non obvious ways. It is now time to draw on the whole continuous torus again, involving pc-sets with different shapes.

\subsection{Other triads}\label{subsection::augmented}
In \cite{CubeDance}, Douthett and Steinbach introduced augmented triads in an effort to rationalize movements between LPR-related triads. How does this fare with the torus model? There is one drawback with the phase measure, namely that the $\arg$ map is undefined in zero. And very regular chords, such as augmented triads, exhibit some nil Fourier coefficients, among which $a_5$. How can we still represent such a chord on the 3-5 torus? Perhaps by considering that the phase of 0 is not undefined, but is {\em any} real number (modulo $2\pi$). Graphically, this means that an augmented triad will be a circle ($a_3$ is fixed, $a_5$ varies arbitrarily) on the torus. Fig. \ref{augmented} is rather neat, with its four horizontal black circles.

\dessin{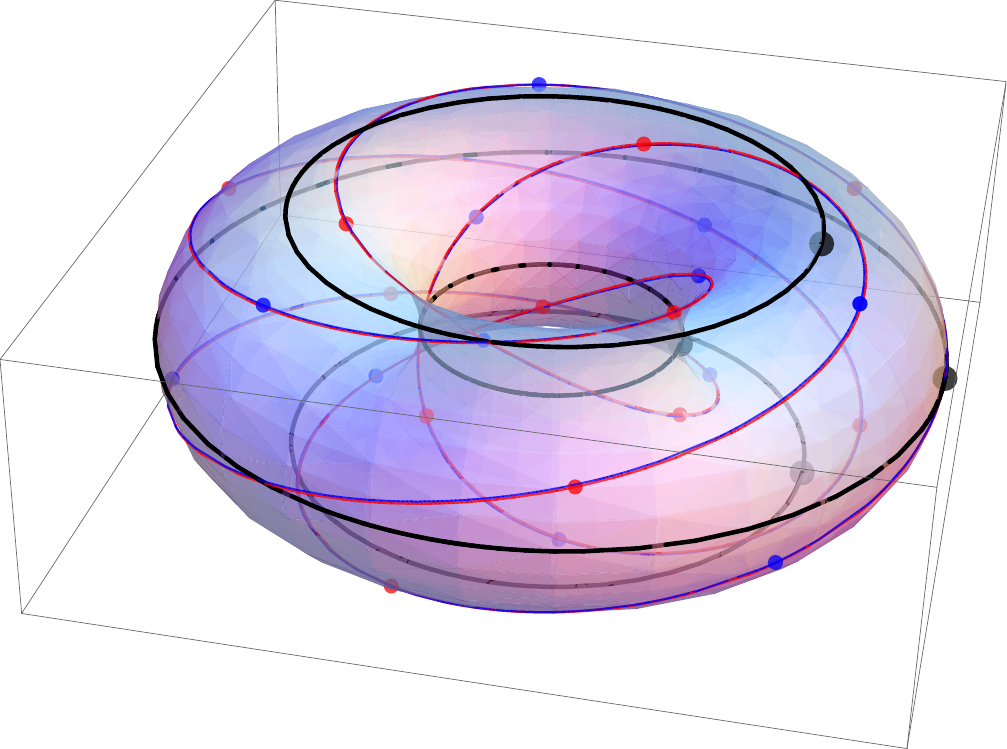}{triads, major, minor and augmented}{12 cm}{augmented}

I think that this picture shows very clearly the pivotal role of an augmented triad (diminished sevenths allow the same kind of modulation versatility), since one can move from one triad -- say D major -- to some close point chosen on a circle representing an augmented triad -- say D aug  -- and from there, circumnavigate the circle and exit close to the next triad -- say G minor, with no distance between one point of the circle and another, since it is really the same chord. The Cube Dance in \cite{CubeDance}, which considers all such moves on a discrete graph, takes place naturally in this picture (it was not drawn here for legibility).

\newpage

However, it could be surmised that {\em diminished} triads, which actually occur inside diatonic collections, would play a more natural role in a torus that is defined with the most important intervals in tonal contexts, e.g. major thirds and fifths (see again \cite{Krumhansl}). This sounds far-fetched, but is actually true. To begin with, diminished triads do not have nil Fourier coefficients with their embarrassing choice of graphic representation. Moreover, the green line connecting all diminished triads in the simplest way manages to wedge itself right between the red and blue lines studied above, so close to one another already (see Fig. \ref{diminished}).

\dessin{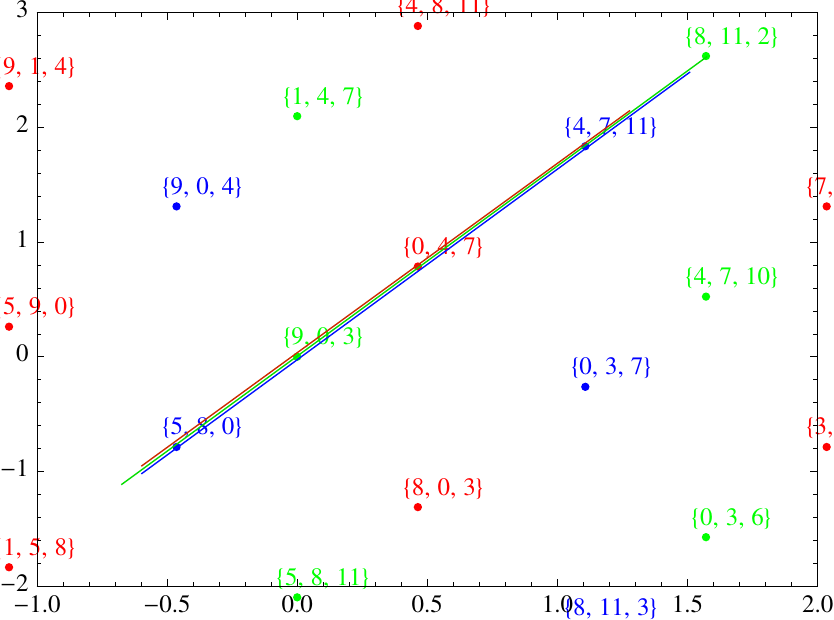}{Triads; major, minor and diminished}{10 cm}{diminished}

Parsimonious sequences of all seven triads included in C major, for instance, draw a sinuous but elegant ragged line in a small region of the torus (see my homepage for a movie, or the dotted line on Fig. \ref{allTriadsCmaj}).

\dessin{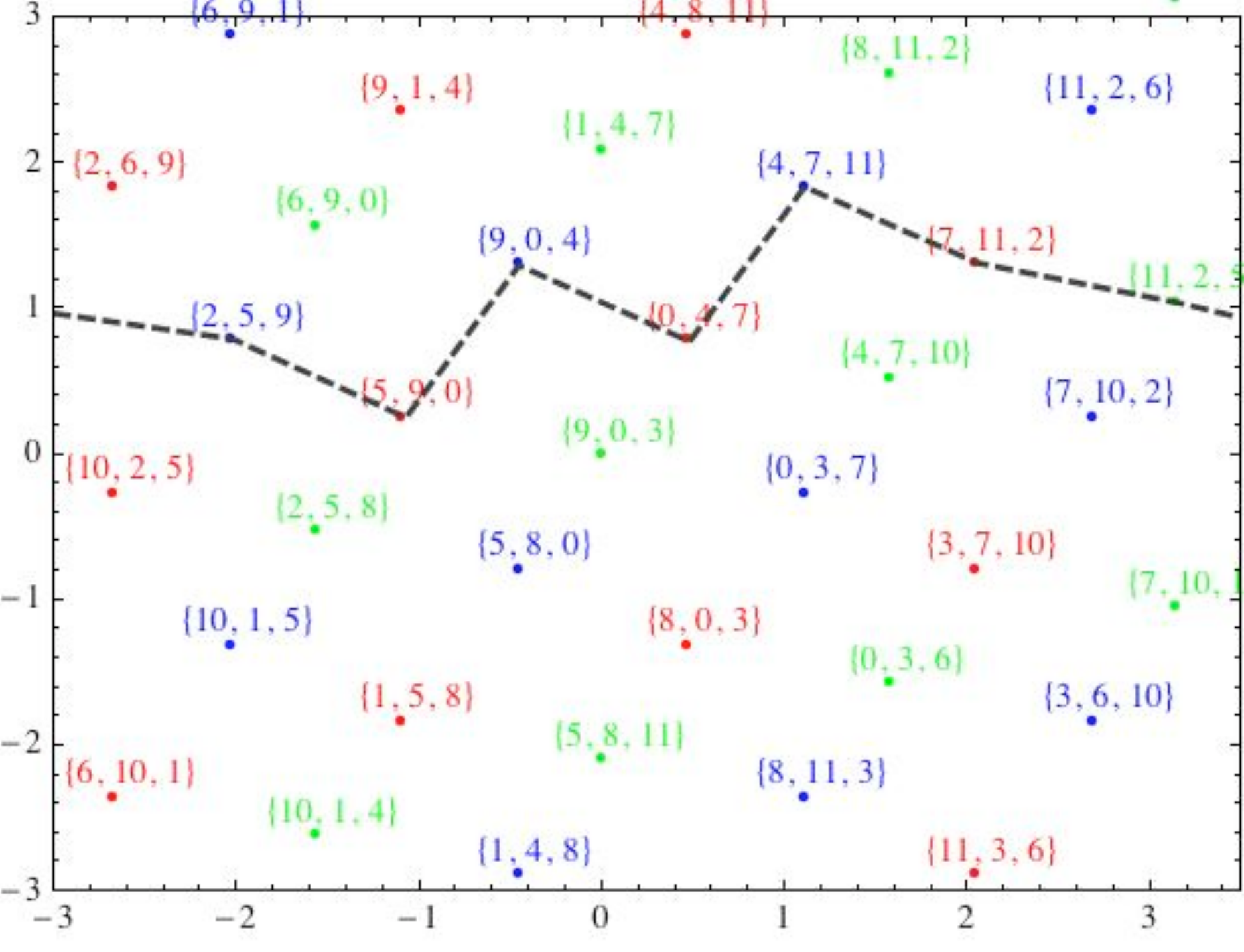}{All seven triads in C major}{10 cm}{allTriadsCmaj}

\subsection{3 or 4 chords}
A vexing question which received a lot of attention in recent research is the connection between chords with different cardinalities, whether using redoubled notes (like the singular points in orbifolds models) or trying for continuous gestures, at least at the generating stage (\cite{Plotkin}). I borrowed a very simple sequence of chords from this last reference (Fig. \ref{Plotkin})

\dessin{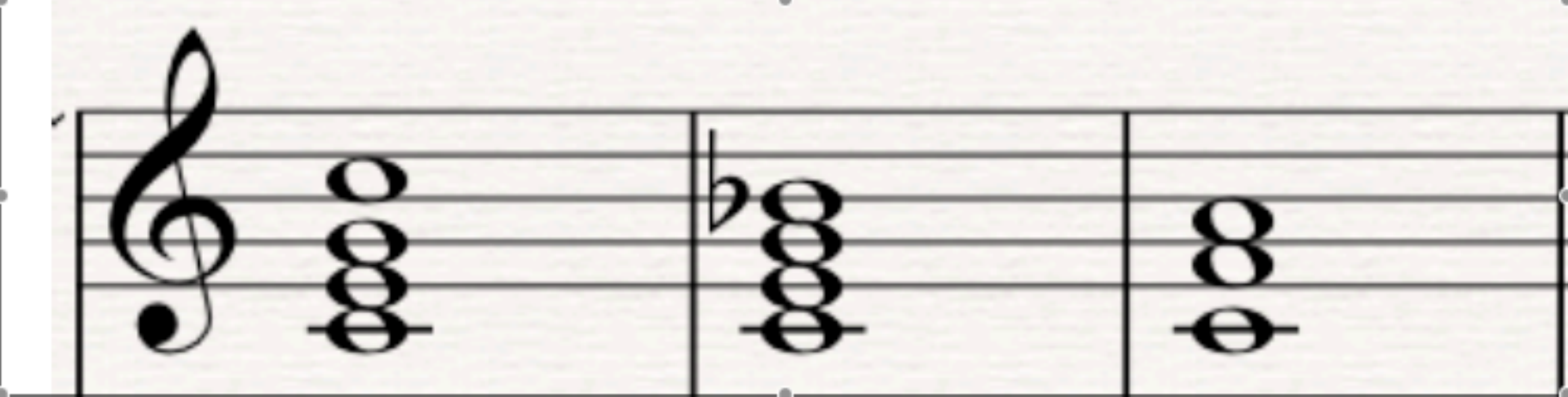}{An ordinary sequence of chords}{10 cm}{Plotkin}

 and computed their angular positions on the 3-5 torus (Fig. \ref{PlotkinTorus}, blown up for legibility). It can be seen that redoubling the tonic in the C major triad does not move the point on the torus away\footnote{ This might be another way to make sense of the elements of the red line between regular triads.} from the red line of major triads (the middle one), while the dominant seventh is clearly distinct from this paradigm, though not drastically so, which expresses rather well the kinship between these different chords. 

\dessin{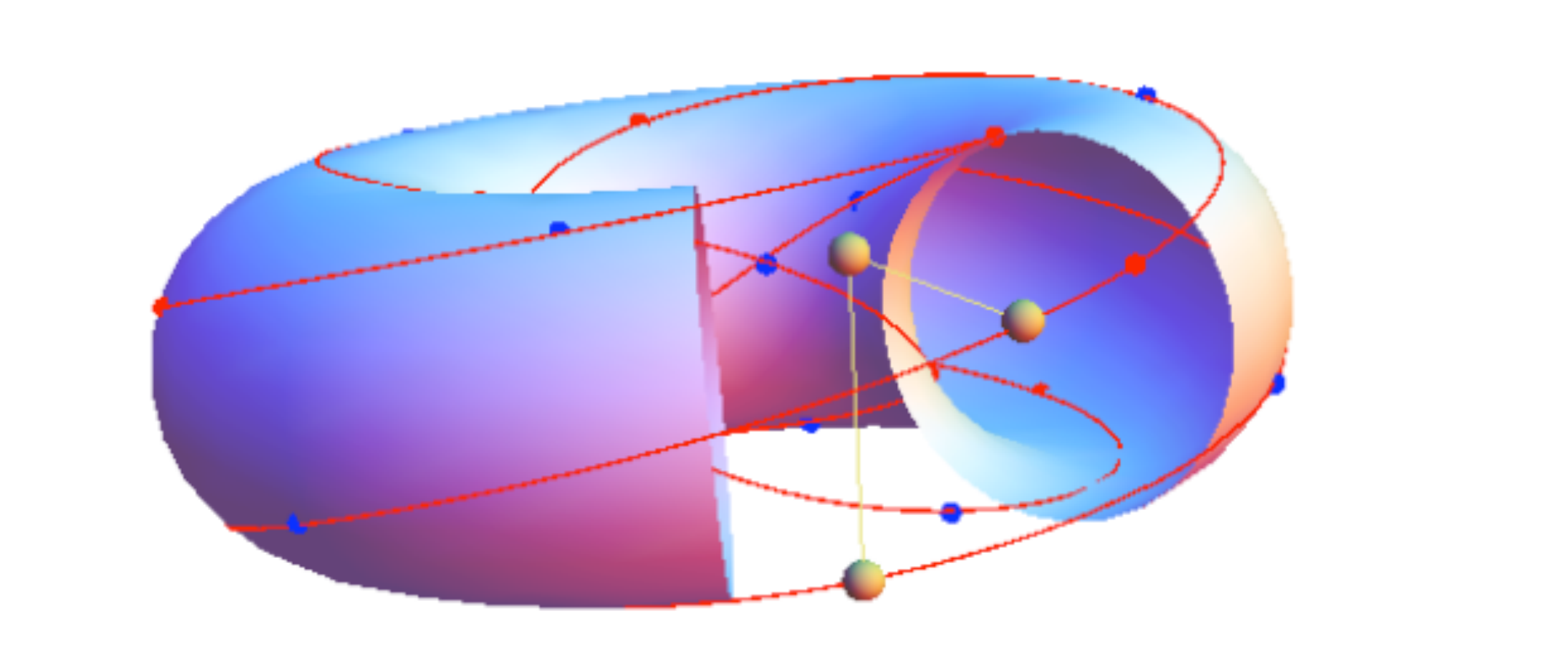}{Same on the torus}{10 cm}{PlotkinTorus}

We will end this sequence of examples with the beautiful and fairly complex beginning of Schumann's {\em Das bittendes Kind} in Kinderszenen opus 27.

\dessin{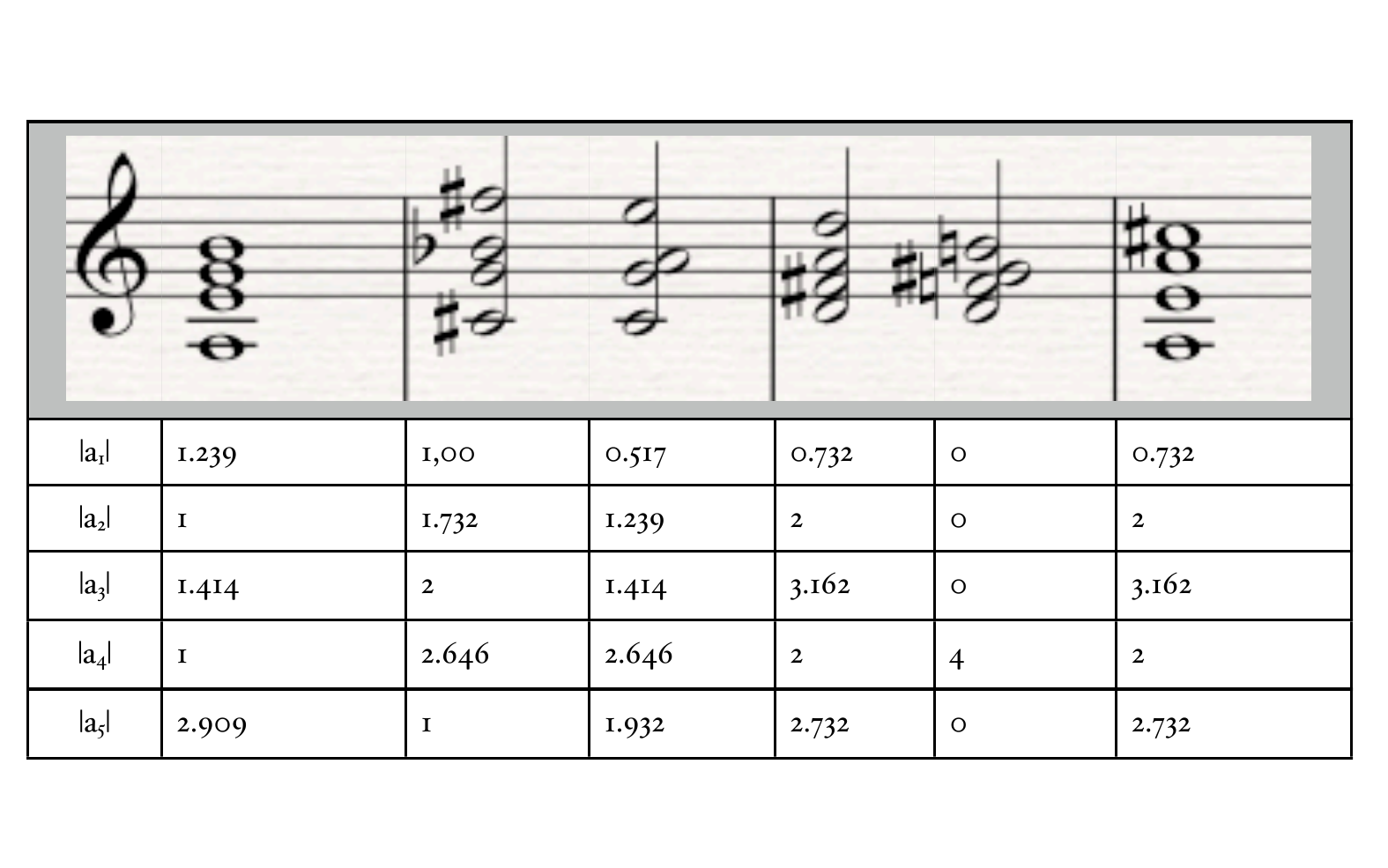}{Das bittendes Kind}{14 cm}{BittendesKind}

Here there is a momentous choice of which Fourier coefficients to consider\footnote{ The 3-5 torus is not a good choice anymore, since for instance the diminished seventh has $a_3 = a_5 = 0$ and could only be represented as suffusing the whole torus.}. The choice of $a_4$ seemed mandatory for 4-chords, for another one it was a close choice between the {\em fifthishness} ($a_5$) and the {\em major thirdishness} ($a_3$) as the most interesting coefficients. The picture shows their magnitudes, which helped select the 4-5 torus with coordinates $(\arg a_4, \arg a_5)$.
The ensuing picture is quite illuminating (Fig. \ref{schumann}).

\dessin{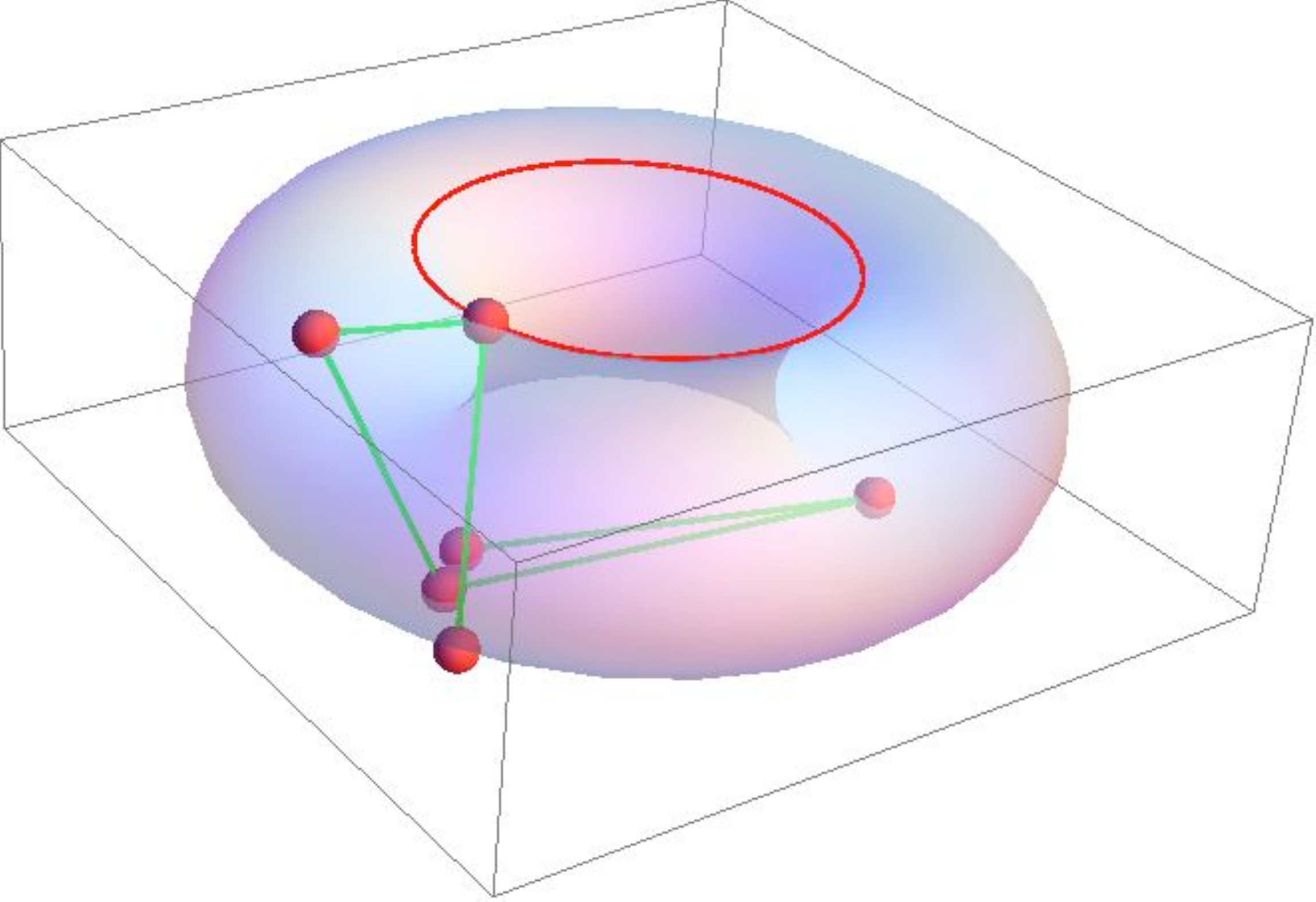}{Fourier coefficients in "das bittendes Kind"}{14 cm}{schumann}

Quite clearly, the weird second chord (far right on the picture) is a long way from either the starting one or the next (or indeed any other). Its two neighbours in time are much closer (a parsimonious move) to one another on this picture, though one (A-E-G-B, the most Þfth-ish) is a kind of incomplete major ninth, and the other an ordinary dominant seventh. Finally, the diminished seventh also departs from the region of ÔordinaryÕ triads with redoubled tonic, though not radically, especially if the argument of the nil $4^{th}$ coefficcient is well chosen (the diminished $7^{th}$ is again a circle, not a point, in this model, cf. the augmented triads in the discussion above) which, I think, reßects rather well the perception of this passage.

\newpage
\section*{Perspectives}
First and foremost, the model outlined above on a handful of musical examples should be tested on a 
larger scale, with numerous chord sequences, confronting it  with `` common musical sense''. To my 
mind, the graphical analysis / real-time representation of diverse musical pieces popularized by 
TymoczkoÕs movies of ChopinÕs E minor prelude or Deep Purple's {\em Smoke on the Water} (\cite{TymoMovies}), and hauled up to a level of graphic artistry in BaroinÕs work (\cite{Baroin}), offers powerful insights into the most intimate relationships between pc-sets. It would be interesting to compare what my phase torii have to offer in that respect.

On the theoretical line, I would like to explore the connections with BaroinÕs `` all chords model''\ (both graphically/geometrically and mathematically), since there are some striking coincidences with my 
results, though his starting points are {\em a priori} completely different (the torus of thirds). All the more reason to try and find relationships. Other interesting researchs involve continuous moves of Fourier coefficients and their repercution in physical space, a most striking example being Noll and Carl\'e's demonstration of `` Fourier scratching of rhythms''\ in the 2009 MCM convention in Yale (\cite{NollDJ}), extended to the realm of musical scales in \cite{ScaleLab}.

I have tried to make sense of fractional multi-pc-sets (generic points on the torus) inasmuch as they 
mimick transpositions of genuine pc-sets (though I must reinforce that microtonal transpositions of pc-
sets do not live within those spaces, whose original dimensions are the twelve pitch classes), or pc-sets with accented notes (truth values greater than 1) as we have seen when discussing Plotkin's example. More generally, it is vital to make better sense of any eventual relationship with the continuous (quotient) 
pitch spaces described by \cite{orbifolds, Callender}. 
One possible way would be to allow inÞnite dimension Hilbert spaces, wherein pitches can move 
continuously (see again \cite{Callender}) alongside the phase parameter: in the Fourier formulasÕ 
exponential coefficients $e^{-2i\pi k t/c}$, $k$ and $t$ would both be real numbers. Such large spaces, which can be projected on both QuinnÕs `` Fourier space''\ and all orbifolds, could be seen as Limits in category theory, a favorite generalization tool in MazzolaÕs framework. Some theorists are averse to those hegelian `` whole spaces of everything'', which might in this context be limited to subspaces / quotients / submanifolds (like the torii above) where a description by Fourier transform is relevant (spectral spaces, as \cite{Fuglede} coined it). 

There are many more possible questions and alleys to explore, but to the author the most pressingly 
puzzling one is: why are the green / blue / red lines so close on the 35 - torus? I think that any answer 
to this question will be essentially musical.

\section*{Thanks}

Gilles Baroin had been toying with Fourier coefficients without knowing it for some time, with his 
wonderful 4D graphic representations of pitch classes and chords. We had enlightening discussions 
about the convergence of our researches and the importance of geometric representations; which 
certainly were a strong motivation for this paper. 

To Pierre Beauguitte for taking up the beacon of Fourier coefficients of pc-sets in his master 
dissertation, with a quick grasp of their essential role and quite a few innovative ideas. It gave me the 
incentive to settle down to work on the topic again.

To Cliff Callender for his beautifully simple way of explaining rather abstruse concepts like Fourier transforms in quotient spaces.

To Jack Douthett who turned scales into algebraic formulas, and carries on in the right direction (that is to say Fourier transforms).

I am grateful to Aline Honingh for showing me from another standpoint (through Carol Krumhansl's 
book) the importance of $3^d$ and $5^{th}$ Fourier coefficients in the perception of tonality, and great work connecting Fourier coefficients and prevalence of interval classes. 

To Guerino Mazzola for his breathtaking formalization of gestures (and his seminal paper does mention 
Fourier space as a potentially interesting example!) and all our fascinating discussions. 

To Thomas Noll, who gently nudged me to look at the meaning of the phase of Fourier coefficients, not 
only their module, during one of our proliÞc strolls in Barcelona; and made some of the most interesting breakthroughs about the use of these coefficients in several unexpected musical domains. 

To Ian Quinn, who exhumated Fourier coefficients out of disuse and rejuvenated them, for our lasting 
delight.

To Bill Sethares for a thoroughly enjoyable trip into simplification of abstract topics.

To Dmitri Tymoczko for many interesting discussions (usually quite lively), innovative ideas, and 
intellectual honesty in investigating the problematic points in the orbifold models and voice-leading 
distances.

To the editor and anonymous reviewers of MTO who provided useful feedback on a previous, more ambitious version of this paper.

And all members of the mathemusical community for their innovative ideas and strong incentives to 
carry on exploring this fascinating research area.


\begin{thebibliography}{99}

\bibitem{AmiotJMM} Amiot, E. 2007. {\em  David Lewin and Maximally Even Sets }. In  Journal of Mathematics and Music, vol 3.

\bibitem{AmiotSethares} Amiot, E., Sethares, B, 2011. {\em An Algebra for periodic rhythms and scales}, JMM (3).

\bibitem{Baroin} Baroin, G., 2011. {\em  The Planet-4D Model: An Original Hypersymmetric Music Space Based on Graph Theory }. In Mathematics and Computation in Music, 3d international conference, Paris, 
Heidelberg: Springer, pp. 326-329.

\bibitem{Callender} Callender, C., 2007. {\em Continuous Harmonic Spaces}, Journal of Music Theory 51.2, pp. 277-332.

\bibitem{CubeDance} Douthett, J. and Steinbach, P. 1998. {\em Parsimonious Graphs: a Study in Parsimony, 
Contextual Transformations, and Modes of Limited Transposition.} Journal of Music Theory 42.2, pp. 241Ð263. 

\bibitem{Fuglede} Fuglede, H., 1974. "Commuting Self-Adjoint Partial Differential Operators and a Group 
Theoretic Problem." J. Func. Anal. 16, pp. 101-121. 

\bibitem{Hoffman} Hoffman, J., 2008. {\em On Pitch-Class Set Cartography Relations between Voice-Leading Spaces and Fourier Spaces }, JMT, 52.2.

\bibitem{Homometry} Mandereau, J., \& alii, 2011. {\em Discrete Phase Retrieval in Musical Structures}, JMM (2).

\bibitem{Krumhansl} Krumhansl, C., 1990. {\em  Cognitive foundations of musical pitch }. Oxford University Press. 

\bibitem{Lewin} Lewin, D., 1987. {\em Generalized Musical Intervals and Transformations}, Yale University Press, New Haven.

\bibitem{orbifolds} Callender, C., Quinn, I., Tymoczko, D, 2008. {\em Generalized Voice Leading Spaces}, Science 320: 346Ð348.

\bibitem{Peck} Peck, B., 2011. {\em $N^{th}$ roots of pitch-class inversion}. In Mathematics and Computation in Music, $3^d$ international conference, Paris, Heidelberg: Springer, pp. 196-206. 

\bibitem{Plotkin} Plotkin, R., 2011. {\em Cardinality Transformations in Diatonic Space}. In Mathematics and Computation in Music, 3d international conference, Paris, Heidelberg: Springer, pp. 207-219. 

 
 \bibitem{Mazzola} Mazzola G., Andreatta M., 2007. {\em Diagrams, Gestures and Formulaes in Music}, JMM (3).

\bibitem{NollDJ} CarlŽ, M., 2009, Hahn, S., Matern, M., Noll, T., {\em Presentation of Fourier Scratching} at MCM 2009, New Haven. See http://www.supercollider2010.de/images/papers/fourier-scratching.pdf

\bibitem{ScaleLab} Milne, A., Carl\'e, M., Sethares, W., Noll, T., Holland, S, 2011. {\em Scratching the Scale Labyrinth}. In Mathematics and Computation in Music, 3d international conference, Paris, Heidelberg: Springer, pp. 207-219.

\bibitem{QuinnPhD} Quinn, I., 2007. {\em A Unified Theory of Chord Quality in Equal Temperaments}, Ph.D. dissertation, Eastman School of Music.

\bibitem{QuinnJMT} Quinn, I., 2007. {\em General Equal Tempered Harmony}, JMT 44 (2) 2006, 45 (1) 2007 . 

\bibitem{TymoczkoJMT} Tymoczko, D., 2008. {\em Set-Class Similarity, Voice Leading, and theÊFourier Transform.} JMT 52.2.

\bibitem{TymoczkoGoM} Tymoczko, D., 2011. {\em The Geometry of Music}. Oxford University Press.

\bibitem{Tymo2010} Tymoczko, D., 2010. {\em Three Conceptions of Musical Distance.} In Mathematics and Computation in Music, eds. Elaine Chew, Adrian Childs, and Ching-Hua Chuan, Heidelberg: Springer, pp. 258Ð273. 

\bibitem{TymoczkoMTO} Tymoczko, D., 2010. {\em Geometrical Methods in Recent Music Theory}, MTO 16 (1), www.mtosmt.org/issues/mto.10.16.1/mto.10.16.1.tymoczko.html

\bibitem{TymoMovies} Tymoczko, D., http://dmitri.tymoczko.com/ChordGeometries.html



\end{thebibliography}
\end{document}